\documentclass[12pt]{amsart}
\numberwithin{equation}{section}
\usepackage{amsmath,amsthm,amsfonts,amscd,eucal,mathabx}

\usepackage{xypic}
\usepackage{graphicx}
\usepackage{amssymb}
\def\cb{{\mathcal B}}

\def\cd{{\mathcal D}}

\def\ch{{\mathcal H}}

\def\cs{{\mathcal S}}

\def\ga{{\mathfrak A}}
\def\gb{{\mathfrak B}}

\def\ba{{\mathbb A}}

\def\bc{{\mathbb C}}

\def\bn{{\mathbb N}}

\def\br{{\mathbb R}}
\def\bt{{\mathbb T}}

\def\bz{{\mathbb Z}}

\def\a{\alpha}
\def\b{\beta}
\def\g{\gamma}  
\def\d{\delta}

\def\l{\lambda} 

\def\m{\mu}

\def\n{\nu}
\def\r{\rho}
\def\s{\sigma} 
\def\t{\tau}
\def\f{\varphi}  \def\F{\Phi}
\def\th{\theta} 
\def\om{\omega}

\def\id{\hbox{id}}
\def\ker{\hbox{Ker}}

\newtheorem{thm}{Theorem}[section]
\newtheorem{lem}[thm]{Lemma}

\newtheorem{cor}[thm]{Corollary}
\newtheorem{prop}[thm]{Proposition}

\newtheorem{rem}[thm]{Remark}
\newtheorem{defin}[thm]{Definition}

\newcommand{\ty}[1]{\mathop{\rm {#1}}}
\def\di{{\rm d}}

\def\ad{\mathop{\rm ad}}

\def\idd{{1}\!\!{\rm I}}

\DeclareMathAlphabet{\mathpzc}{OT1}{pzc}{m}{it}

\begin{document}
\title[the Anzai skew-product for the noncommutative torus]
{Ergodic properties of the Anzai skew-product for the noncommutative torus}
\author[S. Del Vecchio]{Simone Del Vecchio}
\address{Simone Del Vecchio\\
Dipartimento di Matematica \\
Universit\`{a} di Roma Tor Vergata\\
Via della Ricerca Scientifica 1, Roma 00133, Italy} \email{{\tt
delvecch@mat.uniroma2.it}}
\author[F. Fidaleo]{Francesco Fidaleo}
\address{Francesco Fidaleo\\
Dipartimento di Matematica \\
Universit\`{a} di Roma Tor Vergata\\
Via della Ricerca Scientifica 1, Roma 00133, Italy} \email{{\tt
fidaleo@mat.uniroma2.it}}
\author[L. Giorgetti]{Luca Giorgetti}
\address{Luca Giorgetti\\
Department of Mathematics, Vanderbilt University, 1326 Stevenson Center, Nashville, TN 37240, US}
 \email{{\tt
 luca.giorgetti@vanderbilt.edu}}
\author[S. Rossi]{Stefano Rossi}
\address{Stefano Rossi\\
Dipartimento di Matematica \\
Universit\`{a} di Roma Tor Vergata\\
Via della Ricerca Scientifica 1, Roma 00133, Italy} \email{{\tt
rossis@mat.uniroma2.it}}
\date{\today}


\begin{abstract}
\vskip0.1cm\noindent
We provide a systematic study of a noncommutative extension of the classical Anzai skew-product for the cartesian product of two copies of the unit circle to the noncommutative 2-tori. In particular, some relevant ergodic properties are proved for these quantum dynamical systems, extending the corresponding ones enjoyed by the classical Anzai skew-product. 

As an application, for a uniquely ergodic Anzai skew-product 
$\F$ on the noncommutative $2$-torus $\ba_\a$, $\a\in\br$, we investigate the pointwise limit, $\lim_{n\to+\infty}\frac1{n}\sum_{k=0}^{n-1}\l^{-k}\F^k(x)$, for $x\in\ba_\a$ and $\l$ a point in the unit circle, and show that there are examples for which the limit does not exist, even in the weak topology.
\end{abstract}

\maketitle

\section{introduction}
\label{sec1}

The systematic study of noncommutative aspects of many branches of mathematics has seen an impetuous growth in the last decades, not least because of their applications to quantum physics.

Among these aspects, the analysis of the ergodic properties of noncommutative dynamical systems, which are in fact a natural framework for the description of quantum physical systems, is undoubtedly worth mentioning.

The investigation of ergodic properties of classical dynamical systems was at first motivated by the problem of justifying the thermodynamical laws from the microscopic principles of statistical mechanics, {\it i.e.} the so-called ergodic hypothesis. 

Given a classical dynamical system $(X,T,\m)$, where $X$ is a compact topological space, $T:X\rightarrow X$ a continuous map, and $\m$ an invariant probability measure under the action of $T$, classical ergodic theory primarily deals with the study of the long-time behaviour of the Cesaro means ({\it i.e.} ergodic~averages)
\begin{equation}
\label{anl}
M_{f,\l}(n):=\frac1{n}\sum_{k=0}^{n-1}\l^{-k}f\circ T^{k}\,,\quad n\in\bn\,,
\end{equation}
$|\l|=1$, of continuous functions or more generally of any measurable function $f$.

Among the most celebrated classical ergodic theorems, we mention the Birkhoff individual ergodic theorem, which concerns the study of the pointwise limit
$\lim_{n\to+\infty}M_{f,1}(n)(x)$, $x\in X$ when $f$ is summable, and the von Neumann mean 
ergodic theorem, which concerns the limit 
$L^2\!-\!\lim_{n\to+\infty}M_{f,1}(n)$ when $f$ is square-summable. 

The uniform limit of the averages $M_{f,\l}(n)$ was systematically investigated in \cite{R}, showing convergence for $\l$ in a particular subset of the 
unit circle depending on $T$, and examples of non-convergence when $\l$ belongs to its complement.

The amount of results obtained in the commutative setting is far too vast to provide an exhaustive account. A standard reference is
\cite{CFS}. We refer to \cite{Fu1} for several unconventional ergodic theorems, which also play a fundamental role in number theory. 

Among the most studied examples, Anzai skew-products, which made their first appearance in \cite{A}, are a particularly  fine class of dynamical systems, insofar as they are sufficiently general to provide a wide gamut of phenomena, and yet sufficiently restricted to be fully treated. 

Notably, they are automatically uniquely ergodic provided that they are ergodic, but
they may well fail to be so while being minimal.
They are also a major source of inspiration for the present work, which aims to exhibit similar examples to cover genuinely noncommutative situations,
where decidedly fewer examples are known.  

In the framework of quantum physics, it is then natural to address the systematic study of the ergodic properties of quantum ({\it i.e.} noncommutative) dynamical systems. As a matter of fact, the situation in the quantum setting is rather more involved than the classical situation in several respects. For instance, all statements must be provided in terms of the dual concept of ``functions'' instead of ``points'', since the latter are meaningless in the quantum cases.

As for the literature on noncommutative ergodic theory, 
the reader is firstly referred to the seminal paper \cite{NSZ} for a thorough study of the multiple correlations and quantum 
(weak) mixing associated with invariant states with central support in the bidual. We also cite \cite{LongoPel}, where general results on $C^*$-dynamical systems, {\it e.g.} topological transitivity and minimality, are provided. Some natural generalisations of quantum ergodic theory are investigated in a series of  papers \cite{F, F1, F2, F3, F4} without assuming in general the centrality of the support of the involved states, as is done in \cite{NSZ} instead,
whereas the reader is referred to \cite{BF2, CrF, CrF2, F16} for some direct applications to physics and quantum probability. 

Another motivation to study noncommutative algebras comes from Connes' theory of noncommutative geometry, in which
a noncommutative $C^*$-algebra is conceived as the algebra of ``continuous functions" on a certain noncommutative manifold.

Many properties of the manifold under consideration can then be encoded in what is known as a spectral triple, that is the datum of a (possibly noncommutative) $C^*$-algebra represented on some Hilbert space and of an unbounded operator, called the Dirac operator, describing the metric properties of the (classical or quantum) manifold. 
For a comprehensive treatise on noncommutative geometry, we refer the reader to the monograph \cite{C}, the expository paper \cite{CPR}, and the literature cited therein.

It is no coincidence that one of the most studied examples in noncommutative geometry is the noncommutative 2-torus $\ba_\a$ (see {\it e.g.} \cite{B}), since it is a rather simple model, albeit highly nontrivial. It is related to the discrete Canonical Commutation Relations, and it can be considered as a quantum deformation of the classical $2$-torus $\bt^2$, with deformation parameter given by the angle $2\pi\a$, 
$\a\in\mathbb{R}$, entering in the definition of the symplectic form involved in the construction. 

On the $C^*$-algebra $\ba_\a$ describing the noncommutative 2-torus, there is always a faithful trace-state 
$\t$, generalising the Haar-Lebesgue measure $\di\th_1\di\th_2/(2\pi)^2$.
When $\a$ is irrational, the $C^*$-algebra $\ba_\a$ is simple and admits a unique trace, necessarily coinciding with $\t$, whose GNS representation leads to the hyperfinite Murray-von Neumann type $\ty{II_1}$ factor. If $\a$ is rational, $\ba_\a$ is strongly Morita equivalent to the abelian $C^*$-algebra $C(\bt^2)$ of the classical 2-torus $\bt^2$. 

In the setting of noncommutative geometry, type $\ty{III}$ representations of quantum manifolds can be produced by considering twisted spectral triples, see \cite{CM}, which can obtained, {\it e.g.} by deforming ordinary ones by using the Tomita modular operator. 

For the noncommutative 2-torus, we mention the papers \cite{FS, FH}, which focus on the explicit construction type $\ty{III}$ representations of this sort  with the corresponding modular spectral triples, and the basic Fourier analysis for such models.

What the present work aims to do is study a natural generalisation of the Anzai skew-product to the noncommutative 2-torus $\ba_\a$, first introduced by Osaka and Phillips in \cite{OP} with a different name, so as to come up with 
a variety of novel quantum dynamical systems for which  notable ergodic properties can be studied in full detail. 

Since $\ba_\a$ can also be described via the crossed-product construction $C(\bt)\rtimes_{r_{\a}}\bz$ of $C(\bt)$ by the action of the dual group $\widehat{\bt}\sim\bz$ through all powers of the rotations 
$(r_\a f)(z)=f(e^{2\pi\imath\a}z)$ of the angle $2\pi\a$, the construction of the Anzai skew-product for the noncommutative 2-torus should be understood as a first nontrivial, possibly the simplest, generalisation of the classical Anzai skew-product which corresponds to dynamical systems involving merely the classical cartesian product $\bt\times\bt$.

For the reader's convenience, the definition of a noncommutative Anzai skew-product is recalled in full detail in Section \ref{sec3}.
Here we limit ourselves to saying that, associated with an angle $\th$ and a continuous function 
$f:\bt\rightarrow\bt$, there is an automorphism $\F_{\th,f}:\ba_\a\rightarrow\ba_\a$, which is defined on the unitary generators $U$, $V$ of $\mathbb{A}_\alpha$ by
$$
\F_{\th,f}(U)=e^{\imath \th}U\,,\quad \F_{\th,f}(V)=f(U)V\,,
$$
The map $\F_{\th,f}$ is shown to be a $*$-automorphism of $\ba_\a$ leaving the canonical trace $\t$ invariant. We also note that, if $\a=0$ and thus $\ba_\a\sim\bt^2$, such definition coincides with the original one \cite{A}.

Section \ref{sec4} is devoted to the study of the ergodic properties of the above Anzai skew-product. In analogy with what happens in the classical case, if for the angle $\th$ with $\th/2\pi$ irrational, we show that the $C^*$-dynamical system $\big(\ba_\a, \F_{\th,f},\t\big)$ is ergodic if and only if a sequence of cohomological equations, one for each $n\in\bz\smallsetminus\{0\}$ (depending on the deforming angle $\a$) has no non-null measurable solutions. Furthermore, we prove that ergodicity always implies unique ergodicity in compliance with the classical case \cite{Fu}. We then have new nontrivial quantum examples of $C^*$-dynamical systems exhibiting very strong ergodic properties. The reader is referred to \cite{F1} and \cite{FB} for an exhaustive discussion about new aspects, and new examples arising from free probability of these $C^*$-dynamical systems enjoying such a very strong ergodic behaviour, respectively.

However, the strategy of the proof of this last result must be significantly changed as the original proof relies on the notion of generic point (see Lemma 2.1 in \cite{Fu}), which is no longer available  in the noncommutative setting.
Moreover, our proof applies to classical non-separable cases as well, which are dealt with in Section \ref{sec5}. 

Section  \ref{sec6} is devoted to answering a question raised in \cite{F4} whether there exists a genuinely noncommutative uniquely ergodic dynamical system 
for which some weighted Cesaro means, \emph{cf.} \eqref{anl}, fail to converge.

More precisely, suppose we have a uniquely ergodic $C^*$-dynamical system $(\ga,\F,\f)$. Let $\big(\ch_\f,\pi_\f, V_{\f,\F},\xi_\f\big)$ be the covariant GNS representation. In \cite{F4, R}, the inclusion of the pure-point spectra 
$\s_{\mathop{\rm pp}}(\F)\subset\s_{\mathop{\rm pp}}(V_{\f,\F})$ is shown to be, in general,  proper. 
In addition, as in the classical case described in \cite{R}, the corresponding Cesaro means $M_{a,\l}$ analogous to \eqref{anl} converge pointwise in norm for each 
$\l\in\s_{\mathop{\rm pp}}(\F)\bigcup\s_{\mathop{\rm pp}}(V_{\f,\F})^{\rm c}$, yet fail to converge in general, even in the weak topology, for $\l\in\s_{\mathop{\rm pp}}(V_{\f,\F})\smallsetminus\s_{\mathop{\rm pp}}(\F)$.  This is seen  through a counterexample based on the tensor product construction, Proposition 3 in \cite{F4}.
By using the Anzai skew products, we provide counterexamples to the convergence of  $M_{a,\l}(n)$ for cases based on the crossed-product construction $\ba_\a\sim C(\bt)\rtimes_{r_{\a}}\bz$.

\noindent

\section{preliminaries}
\label{sec2}

\subsection{Notations}

Let $E$ be a normed space. We simply denote by $\|\,{\bf\cdot}\,\|$ its norm whenever no confusion arises. In particular, $\|x\|\equiv\|x\|_\ch$ will be the Hilbertian norm of an element $x$ in the Hilbert space $\ch$. For a continuous or measurable function $f$ defined on the locally compact space $X$ equipped with the Radon measure $\m$, $\|f\|\equiv\|f\|_\infty$ will denote the ``esssup'' norm (or just the "sup" norm  for $f$ continuous) of the function $f$. 

In such a situation, $C_{\rm b}(X)$ will denote the $C^*$-algebra consisting of all bounded continuous functions defined on $X$, equipped with the natural algebraic operations, and norm 
$\|f\|\equiv\|f\|_\infty:=\sup_{x\in X}|f(x)|$. The situation of a point-set $X$, together with the $C^*$-algebra $\cb(X)=C_{\rm b}(X)$ of all bounded functions defined on $X$, is reduced to 
a particular case of the previous one by considering $X$ equipped with the discrete topology.

\vskip.3cm

Let $\bt:=\{z\in\bc\mid |z|=1\}$ be the abelian group consisting of the unit circle. It is homeomorphic to the interval $[0,2\pi)$ by means of the map $[0,2\pi)\ni\a\mapsto e^{\imath\a}\in\bt$, after identifying the end-points $0$ and $2\pi$. We equivalently use both descriptions without further mention.
Henceforth, for $S\subset\bt$, $S^{\rm c}$ will denote $\bt\smallsetminus S$.

The dual topological group $\widehat{\bt}$ is isomorphic to the discrete group $\bz$. The corresponding Haar measures are the normalised Lebesgue measure $\di{\rm m}=\frac{\di z}{2\pi\imath z}=\frac{\di \th}{2\pi}$ for $z=e^{\imath\th}$ on the circle, and the counting measure on $\bz$, respectively. 

The Fourier transforms of a bounded signed Radon measure $\m\in C(\bt)^*$, and a summable function $f\in L^1(\bt,{\rm m})$, are respectively defined as
$$
\widehat{\m}(n):=\int_\bt z^{-n}\di\m(z)\,,\,\,\, \widehat{f}(n):=\oint f(z)z^{-n}\frac{\di z}{2\pi\imath z}\,,\quad n\in\bz.
$$

\vskip.3cm

If $\mathpzc{f}:X\rightarrow X$ is a map on a point-set $X$,
\begin{itemize} 
\item[{\bf --}] $\mathpzc{f}^0:=\id_X$, and 
for the $n$-times composition, $\mathpzc{f}^n:=\underbrace{\mathpzc{f}\circ\cdots\circ \mathpzc{f}}_{n-\text{times}}$.
\end{itemize}
Therefore, $\mathpzc{f}^{n}:X\rightarrow X$, $n\in\bn$, defines an action of the monoid $\bn$ on the set $X$. 

If in addition $\mathpzc{f}:X\rightarrow X$ is invertible,
\begin{itemize} 
\item[{\bf --}] for the $n$-times composition of the inverse, $\mathpzc{f}^{-n}:=\underbrace{\mathpzc{f}^{-1}\circ\cdots\circ \mathpzc{f}^{-1}}_{n-\text{times}}$.
\end{itemize}
Therefore, $\mathpzc{f}^{n}:X\rightarrow X$ is meaningful for any $n\in\bz$ and provides an action of the group $\bz$ on the set $X$. 

Concerning the cases under consideration relative to the noncommutative 2-torus, for $\a\in\br$ we denote by $R_\a$ the rotation of the angle $2\pi\a$ acting on the unit circle:
$$
\bt\ni z\mapsto R_\a z:=e^{2\pi\imath\a}z\in\bt\,.
$$

\subsection{$C^*$-dynamical systems}
 
A (discrete) $C^*$-dynamical system is a triple $(\ga,\F,\f)$ consisting of a $C^*$-algebra, a positive map $\F:\ga\rightarrow\ga$, and a state $\f\in\cs(\ga)$ such that $\f\circ\F=\f$, and
$$
\f\big(\F(a)^*\F(a)\big)\leq \f(a^*a)\,,\quad a\in\ga\,.
$$
Consider the Gelfand--Naimark--Segal (GNS for short) representation $\big(\ch_\f,\pi_\f,\xi_\f\big)$, see {\it e.g.}~\cite{Sa}. 
Then there exists a unique linear contraction $V_{\f,\F}\in\cb(\ch_\f)$ such that $V_{\f,\F}\xi_\f=\xi_\f$ and
$$
V_{\f,\F}\pi_\f(a)\xi_\f=\pi_\f(\F(a))\xi_\f\,,\quad a\in\ga\,.
$$
The quadruple $\big(\ch_\f,\pi_\f, V_{\f,\F},\xi_\f\big)$ is called the {\it covariant GNS representation} associated with the triple $(\ga,\F,\f)$.

If $\F$ is multiplicative, hence a $*$-homomorphism, then $V_{\f,\F}$ is an isometry with final range $V_{\f,\F}V_{\f,\F}^*$, the orthogonal projection onto the subspace $\overline{\pi_\f(\F(\ga))\xi_\f}$, see {\it e.g.} \cite{NSZ}, Lemma 2.1.

For the $C^*$-dynamical system $(\ga,\F,\f)$, the case when $\ga$ is a unital $C^*$-algebra with unity \mbox{$\idd\equiv\idd_\ga$}, and $\F$ is multiplicative and identity-preserving, {\it i.e.} a unital $*$-homomorphism, is of primary importance. 
Indeed, denote by $\ga^\F:=\big\{a\in\ga\mid\F(a)=a\big\}$ the fixed-point subalgebra, and
$$
\s^{\mathop{\rm ph}}_{\mathop{\rm pp}}(\F):=\big\{\l\in\bt\mid \l\,\text{is an eigenvalue of}\,\F\big\}
$$
the set of the peripheral eigenvalues of $\F$ ({\it i.e. the peripheral pure-point spectrum}) with $\ga_\l$ the relative eigenspaces. Obviously, $\idd\in\ga^\F=\ga_1$.

For $\xi\in\ch_\f$ and $n\in\bz$, consider the sequence
$$
\widehat{\m_\xi}(n):=\bigg\{\begin{array}{ll}
          \langle V^n_{\f,\F}\xi,\xi\rangle& \text{if}\,\, n\geq0\,, \\[1ex]

       \overline{ \langle V^{-n}_{\f,\F}\xi,\xi\rangle} & \text{if}\,\, n<0\,.
         \end{array}
         \bigg.
$$
It is easily seen that the sequence $\big\{\widehat{\m_\xi}(n)\big\}_{n\in\bz}$ is positive definite, and therefore it is the Fourier transform of a positive bounded Radon measure $\m_\xi$ on the unit circle $\bt$, see {\it e.g.} \cite{F4}, Proposition 1.

Equally well, we can consider the peripheral pure-point spectrum 
$$
\s^{\mathop{\rm ph}}_{\mathop{\rm pp}}(V_{\f,\F}):=\big\{\l\in\bt\mid \l\,\text{is an eigenvalue of}\,V_{\f,\F}\big\}
$$ 
of $V_{\f,\F}$.\footnote{If $\F:\ga\rightarrow\ga$ is a $*$-automorphism leaving $\f\in\cs(\ga)$ invariant, $\s(\F)$ and $\s(V_{\f,\F})$ belong to $\bt$. Therefore, 
$\s^{\mathop{\rm ph}}_{\mathop{\rm pp}}(\F)=\s_{\mathop{\rm pp}}(\F)$ and
$\s^{\mathop{\rm ph}}_{\mathop{\rm pp}}(V_{\f,\F})=\s_{\mathop{\rm pp}}(V_{\f,\F})$.}
Denote with $E^{\f,\F}_\l\in\cb(\ch_\f)$ the orthogonal projection onto the eigenspace generated by the eigenvectors associated with $\l\in\bt$, with the convention $E^{\f,\F}_\l=0$ if 
$\l\notin\s^{\mathop{\rm ph}}_{\mathop{\rm pp}}(V_{\f,\F})$. With an abuse of language, $\m_\xi$ is the spectral measure of $V_{\f,\F}$ relative to $\xi\in\ch_\f$. 
Therefore, if~$\l=e^{-\imath\th}\in\bt$, then
$\m_\xi(\{\th\})=\|E^{\f,\F}_\l\xi\|^2$.

The $C^*$-dynamical system $(\ga,\F)$ made of a unital $C^*$-algebra $\ga$ and an identity-preserving completely positive map $\F:\ga\rightarrow\ga$ is said to be:
\begin{itemize}
\item[-] {\it topologically ergodic} if $\ga^{\F}=\bc\idd_\ga$;
\item[-] {\it minimal} ({\it cf.} \cite{LongoPel}) if $\ga$ itself is the unique nontrivial $\F$-invariant hereditary $C^*$-subalgebra of $\ga$;
\item[-] {\it uniquely ergodic} if there exists only one invariant state $\f$ for the dynamics induced by $\F$. 
\end{itemize}
For a uniquely ergodic $C^*$-dynamical system, we simply write $(\ga,\F,\f)$, where $\f\in\cs(\ga)$ is the unique invariant state. 

Since we are interested in $*$-automorphisms, we specialise the situation to the case when $\F$ is a unital $*$-homomorphism of the unital $C^*$-algebra $\ga$.

The commutative case can be considered as a particular case of the previous ones. Namely, with an abuse of notation, a classical $C^*$-dynamical system will be a pair $(X,T)$ consisting of a 
locally compact Hausdorff space $X$ and a continuous map $T:X\rightarrow X$. An invariant state will be an invariant Radon probability measure $\m$ on $X$, that is $\m\circ T^{-1}=\m$. The unital case corresponds to $X$ being compact, whereas the automorphic action cases correspond to $T$ being a homeomorphism. In this situation, the underlying $C^*$-algebra $\ga$, the positive map
$\F$ and the invariant state $\om$ correspond to $C_o(X)$, the set of all continuous functions vanishing at infinity, $\F(f)=f\circ T$, and $\om(f)=\int_Xf\di\m$ for $f\in C_o(X)$, respectively.

Concerning the classical unital cases, for a compact Hausdorff space $X$, a homeomorphism $T:X\rightarrow X$ is called minimal if the unique nontrivial closed set of $X$, invariant under $T$, is $X$ itself.
It is called uniquely ergodic if it admits a unique invariant probability measure $\m\in C(\bt)^*_+$. 

When we deal with $C^*$-dynamical systems associated with abelian $C^*$-algebras as before, we simply refer to a {\it dynamical system}.

\subsection{The noncommutative 2-torus}

We fix the noncommutative 2-torus based on the deformation of the classical 2-torus $\bt^2$, corresponding to the rotation by the angle $2\pi\a$.

Indeed, for a fixed $\a\in\br$, the {\it noncommutative torus} $\ba_{\a}$ associated with the rotation by the angle $2\pi\a$, is the universal $C^*$-algebra with identity $I=\idd_{\ba_\a}$ generated by the commutation relations involving two unitary indeterminates $U,V$:
\begin{equation}
\label{ccrba}
\begin{split}
&UU^*=U^*U=I=VV^*=V^*V,\\
&UV=e^{2\pi\imath\a}VU.
\end{split}
\end{equation}
Such a $C^*$-algebra $\ba_\a$ has a canonical faithful tracial state $\t$, defined on the  total set $\{U^mV^n\mid m,n\in\bz\}\subset\ba_\a$ by
$$
\t(U^mV^n):=\d_{m,0}\d_{n,0}\,,\quad m,n\in\bz\,.
$$
Since the abelian $C^*$-algebra $C^*(U,I)$ provides a copy of $C(\bt)$ inside $\ba_\a$ ({\it cf.} \cite{D}), $\ba_\a$ can be viewed as a crossed-product of $C(\bt)$ by the action of the dual group $\widehat{\bt}\sim\bz$ through the (all powers of the) rotations of the angle $2\pi\a$: $\ba_\a\sim C(\bt)\rtimes_{r_{\a}}\bz$, 
where, with an abuse of notations, $r_{\a}g(z)=g(R_\a z)=g(e^{2\pi\imath\a}z)$, $g\in C(\bt)$.

Let $\big(\ch_\t,\pi_\t,\xi_\t\big)$ be the GNS representation associated with the canonical trace $\t$. We consider the sequence of vectors $\{e_{mn}\mid m,n\in\bz\}\subset\ch_\t$, where
$$
e_{mn}:=\pi_\t\big(U^mV^n\big)\xi_\t\,,\quad  m,n\in\bz\,.
$$
The above sequence is made of orthonormal vectors: 
$$
\langle e_{kl},e_{mn}\rangle=\t\big(V^{-n}U^{k-m}V^l\big)=\d_{k,m}\d_{l,n}\,,
$$
which is indeed a basis of $\ch_\t$.

Among the most relevant properties of the noncommutative 2-torus, we mention the equivalence of the following statements:
\begin{itemize}
\item[(i)] $\a$ is irrational,
\item[(ii)] $\ba_\a$ is simple,
\item[(iii)] $\ba_\a$ has a unique trace, which necessarily coincides with $\t$,
\end{itemize}
see {\it cf.} \cite{B}, Theorem 1.10. In addition,
$$
\ba_{\a_1}\sim\ba_{\a_2}\iff \a_1=\pm\a_2\,\,{\rm mod}\,\,\bz\,.
$$

When $\a$ is irrational, the GNS representation $\pi_\t$ generates the hyperfinite Murray-von Neumann type $\ty{II_1}$ factor.

When $\a$ is rational instead, it is known that $\ba_\a$ is strongly Morita equivalent to $\ba_0\sim C(\bt^2)$, see {\it e.g.} \cite{Ri}.\footnote{Two $C^*$-algebras $\ga$ and $\gb$ are said to be {\it strongly Morita equivalent} if there exist a Hilbert $(\gb,\ga)$-bimodule $E$, and a Hilbert $(\ga,\gb)$-bimodule $F$ such that
$$
E\otimes_\ga F\cong\gb\,,\,\,\,F\otimes_\gb E\cong\ga\,.
$$}

The maps $\r^{(o)}$ are useful in the sequel. They are defined on the generators by
\begin{equation}
\label{rozero}
\r^{(o)}_{z,w}(U)=zU\,,\quad \r^{(o)}_{z,w}(V)=wV\,,\quad z,w\in\bt\,,
\end{equation}
and extend to $*$-automorphisms $\r^{(o)}_{z,w}:\ba_\a\rightarrow\ba_\a$.
In addition, the $\r^{(o)}_{z,w}$ leave the canonical trace $\t$ invariant. 

It is matter of routine to check that 
$$
E_U(x):=\oint\r^{(o)}_{1,z}(x)\frac{\di z}{2\pi\imath z}\,,\quad x\in\ba_\a
$$ 
is a conditional expectation onto $C^*(U,I)\sim C(\bt)$. 
Therefore, each $x\in\ba_\a$, admits an expansion as
\begin{equation}
\label{exfp}
x=\sum_{n\in\mathbb{Z}}c_{n, x}(U)V^n\,,
\end{equation}
where $c_{n, x}\in C(\bt)$ is the unique function such that $c_{n, x}(U)= E_U(xV^{-n})$, and the convergence of the series is understood in norm in the sense of Fej\'er, see {\it e.g.} \cite{D}, Theorem VIII.2.2. 
We also note that $x=0$ if and only if the $c_{n, x}$ are all identically zero.

\subsection{The Anzai skew-product on $\bt^2$}

We recall the definition and the main ergodic properties of the so-called {\it Anzai skew-product} 
$$
\bt\times\bt\ni (s,t)\mapsto T_{\th,h}(s,t)\in\bt\times\bt
$$
associated with the angle $2\pi\th$, where we are tacitly assuming that $\th\in(0,1)$ is irrational, and the continuous function $h:\bt\rightarrow\bt$. Such a map is defined as
\begin{equation}
\label{Anzai}
T_{\th,h}(s,t):=\big(R_\th s, h(s)+ t\big)\,,\quad s,t\in[0,2\pi)\,,
\end{equation}
and the arithmetic is understood mod-$2\pi$.

It is possible to show (see below, or \cite{Fu} for the case when $X$ is a compact metric space) that the Anzai skew-product $T_{\th,h}$ is uniquely ergodic, provided the dynamical system
$\big(\bt^2,T_{\th,h},\di s\di t/4\pi^2\big)$ 
is ergodic.

Consider the {\it cohomological equation}
\begin{equation}
\label{Anzai1}
nh(s)=g(R_\th s)-g(s)\,,\quad n\in\bz\smallsetminus\{0\}
\end{equation}
$\di s/2\pi$-{\it a.e.}, where $g$ is a measurable function.

In \cite{Fu}, it was shown that an Anzai skew-product \eqref{Anzai} is minimal (respectively uniquely ergodic), if and only if the cohomological equation \eqref{Anzai1} has no continuous
(respectively measurable) solutions for any nonzero $n\in\bz$. In particular, any uniquely ergodic Anzai skew-product is necessarily minimal. On the other hand, it is easy to exhibit a uniquely ergodic homeomorphism which is not minimal by considering the one-point compactification $X=\bz\bigsqcup\{\infty\}$ of the integers, and for $T$ the one-step shift (with the convention $T(\infty)=\infty$). 

In \cite{Fu}, an Anzai skew-product was also constructed for which the cohomological equation admits only nontrivial measurable solutions, that is it is minimal but not uniquely ergodic.

\subsection{Uniform convergence of Cesaro means for uniquely ergodic dynamical systems}

For the sake of completeness, we collect some standard but crucial results ({\it cf.} \cite{F4, R}) useful in the sequel.
\begin{prop}
\label{gnsc}
Let $(\ga,\F,\f)$ be a uniquely ergodic $C^*$-dynamical system, and $\{\om_n\}_{n\in\bn}\subset\cs(\ga)$ any sequence of states. Then the following assertions hold true.
\begin{itemize}
\item[(i)] $\s^{\mathop{\rm ph}}_{\mathop{\rm pp}}(\F)$ is a subgroup of $\bt$, and the corresponding eigenspaces $\ga_\l$, 
$\l\in\s^{\mathop{\rm ph}}_{\mathop{\rm pp}}(\F)$, have dimension one and are generated by a single unitary~$u_\l$.
\item[(ii)] $\s^{\mathop{\rm ph}}_{\mathop{\rm pp}}(\F)\subset\s^{\mathop{\rm ph}}_{\mathop{\rm pp}}(V_{\f,\F})$.
\item[(iii)] For each $a\in\ga$ and $\l=e^{-\imath\th}$,
$$
\m_{\pi_\f(a)\xi_\f}(\{\th\})^{1/2}\geq\limsup_n\frac1{n}\bigg|\sum_{k=0}^{n-1}\l^{-k}\om_n(\F^k(a))\bigg|\,.
$$
\end{itemize}
\end{prop}

\noindent\hspace{12.33 cm}$\square$

Concerning the uniquely ergodic dynamical systems, in analogy to the classical situation ({\it cf.} \cite{R}), we call the eigenvectors corresponding to $\l\in\s^{\mathop{\rm ph}}_{\mathop{\rm pp}}(\F)$ {\it continuous eigenvectors}, whereas those corresponding  to
$\l\in\s^{\mathop{\rm ph}}_{\mathop{\rm pp}}(V_{\f,\F})\smallsetminus\s^{\mathop{\rm ph}}_{\mathop{\rm pp}}(\F)$ {\it measurable non-continuous eigenvectors}.

For $a\in\ga$ and $\l\in\bt$, consider the Cesaro averages
\begin{equation}
\label{cesmn}
M_{a,\l}(n):=\frac1{n}\sum_{k=0}^{n-1}\l^{-k}\F^k(a)\,,\quad n\in\bn\smallsetminus\{0\}\,.
\end{equation}
For the convenience of the reader, we include the proof of the main result ({\it cf.} \cite{F4, R}) concerning the uniform convergence of the Cesaro means \eqref{cesmn} 
for uniquely ergodic $C^*$-dynamical systems.
\begin{thm}
\label{main}
Let $(\ga,\F,\f)$ be a uniquely ergodic $C^*$-dynamical system. Fix 
$\l\in\s^{\mathop{\rm ph}}_{\mathop{\rm pp}}(\F)\bigcup\s^{\mathop{\rm ph}}_{\mathop{\rm pp}}(V_{\f,\F})^{\rm c}$. Then for each $a\in\ga$,
$$
\lim_nM_{a,\l}(n)=\f(u_\l^*a)u_\l\,,
$$
uniformly, where $u_\l\in\ga_\l$ is any eigenvector corresponding to $\l\in\s^{\mathop{\rm ph}}_{\mathop{\rm pp}}(\F)$ (with the convention that if 
$\l\in\s^{\mathop{\rm ph}}_{\mathop{\rm pp}}(V_{\f,\F})^{\rm c}$, then $u_\l=0$).
\end{thm}
\begin{proof}
By (ii) in Proposition \ref{gnsc}, $\s^{\mathop{\rm ph}}_{\mathop{\rm pp}}(\F)\subset\s^{\mathop{\rm ph}}_{\mathop{\rm pp}}(V_{\f,\F})$. Therefore, we can
firstly consider the case $\l\in\s^{\mathop{\rm ph}}_{\mathop{\rm pp}}(\F)$ and take a unitary eigenvector $u_\l\in\ga_\l$. 

Since $\F$ is multiplicative, we have
$$
\f(u^*_\l a)u_\l=u_\l\lim_n\bigg(\frac1{n}\sum_{k=0}^{n-1}\F^k(u^*_\l a)\bigg)
=\lim_n\bigg(\frac1{n}\sum_{k=0}^{n-1}\l^{-k}\F^k(a)\bigg)\,.
$$
Let now $\l\in\s^{\mathop{\rm ph}}_{\mathop{\rm pp}}(V_{\f,\F})^{\rm c}$, and suppose $\frac1{n}\sum_{k=0}^{n-1}\l^{-k}\F^k(a)\nrightarrow0$ uniformly. Then there would exist a sequence of states 
$\{\om_n\}_{n\in\bn}\subset\cs(\ga)$ such that for $\l=e^{-\imath\th}$, $\limsup_n\frac1{n}\bigg|\sum_{k=0}^{n-1}\l^{-k}\om_n\big(\F^k(a)\big)\bigg|>0$. By (iii) in Proposition \ref{gnsc},
$$
\m_{\pi_\f(a)\xi_\f}(\{\th\})^{1/2}\geq\limsup_n\frac1{n}\bigg|\sum_{k=0}^{n-1}\l^{-k}\om_n\big(\F^k(a)\big)\bigg|>0
$$
which contradicts $\l\in\s^{\mathop{\rm ph}}_{\mathop{\rm pp}}(V_{\f,\F})^{\rm c}$.
\end{proof}
\begin{cor}
For a uniquely ergodic $C^*$-dynamical system $(\ga,\F,\f)$, if $\l\in\s^{\mathop{\rm ph}}_{\mathop{\rm pp}}(\F)\smallsetminus\{1\}$, with $x\in\ga_\l$ we get $\f(x)=0$. 
\end{cor}
\begin{proof}
By the previous theorem, we know that 
$$
\f(u_\l^*)u_\l=\lim_{n\to+\infty}M_{\idd_\ga,\l}(n)=0\,,
$$ 
and therefore, for $x=\a u_\l$, $\a\in\bc$, we get 
$$
\f(x)\idd_\ga=\a\big(\f(u_\l^*)u_\l\big)^*u_\l=0\,.
$$
\end{proof}
In \cite{F4}, we can find an example based on the tensor product construction, for which the average \eqref{cesmn} fails to converge, even in the weak topology, for some $a\in\ga$ and
$\l\in\s^{\mathop{\rm ph}}_{\mathop{\rm pp}}(V_{\f,\F})\smallsetminus\s^{\mathop{\rm ph}}_{\mathop{\rm pp}}(\F)$. 

In the present paper, we will construct more complicated examples involving the noncommutative 2-torus (hence based on the crossed-product construction), for which the average \eqref{cesmn}
does not converge either.

\section{the anzai skew-product for the noncommutative torus}
\label{sec3}

We denote by $C(\bt;\bt)$ the set consisting of all continuous functions $f:\bt\rightarrow\bt$.

We fix $f\in C(\bt;\bt)$ and an angle $\th\in\br$ (with the arithmetic understood mod-$2\pi$, and $f=e^{\imath h}$ for $h$ in \eqref{Anzai}). For $U$, $V$ the universal generators of $\ba_\a$, we set
\begin{equation}
\label{nazai}
\F_{\th,f}(U):=e^{i\theta}U\,,\quad \F_{\th,f}(V):=f(U)V\,.
\end{equation}
To the best of our knowledge, the above class of automorphisms appeared for the first time in \cite{OP}, where they
are referred to as Furstenberg transformations. The definition given there, though, slightly differs from our definition, 
which is better suited to the present context. In fact, in the aforementioned paper most of the attention is lavished on
the so-called tracial Rokhlin property, while in this paper it is ergodic properties to be dealt with.

\medskip

There follows a couple of known results whose proofs are nevertheless included for the sake of completeness.  
\begin{prop}
\label{nazai1}
Equation \eqref{nazai} uniquely defines a $*$-automorphism of $\ba_\a$, whose inverse is $\F_{\th,f}^{-1}=\F_{-\th,\overline{f\circ R_{-\th/2\pi}}}$.
\end{prop}
\begin{proof}
We start by computing
\begin{align*}
\F_{\th,f}(U)\F_{\th,f}(V)=&e^{\imath\th}Uf(U)V=e^{\imath\th}f(U)UV\\
=&e^{\imath\th}f(U)VUe^{2\pi\imath\alpha}\\
=&\left(f(U)V\right)\left(e^{\imath\th}U\right)e^{2\pi\imath\alpha}\\
=&\F_{\th,f}(V)\F_{\th,f}(U)e^{2\pi\imath\alpha}\,,
\end{align*}
that is the unitaries $\F_{\th,f}(U)$ and $\F_{\th,f}(V)$ satisfy the same commutation relations as $U$ and $V$. Therefore, $\F_{\th,f}$ extends to a $*$-homomorphism of $\ba_\a$
by universality, see {\it e.g.} \cite{B}. 

Finally, it is a matter of computations to see that its inverse is given by
$\F_{\th,f}^{-1}(U)=e^{-\imath\theta}U$ and $\F_{\th,f}^{-1}(V)=f(e^{-\imath\theta}U)^*V=\overline{f\circ R_{-\th/2\pi}}(U)V$, and the proof follows.
\end{proof}
\begin{defin}
\label{nazai2}
The $*$-automorphism of $\ba_\a$ uniquely defined by \eqref{nazai} will be called the Anzai skew-product for the noncommutative 2-torus.
\end{defin}
\begin{rem}
We point out the following facts:
\begin{itemize}
\item[(i)] $f(U)^*$ will in general fail to coincide with $f(U^*)$. However, they do coincide if and only if
$f(\bar{z})=\overline{f(z)}$, for any $z\in\mathbb{T}$. 
\item[(ii)] For the Anzai skew-product in $\ba_\a$, it would have been possible to define $\F(U)=e^{i\theta}U$ and $\F(V)=Vf(U)$ or, equally well, any other reasonable interpolated definition between $f(U)V$ and $Vf(U)$ for $\F(V)$. By using \eqref{ccrba}, it is easily seen that
$$
Vf(U)=\F_{\th,f\circ R_{-\a}}(V)\,,
$$
and therefore we can interplay between both definitions by changing the continuous function $f$ in a trivial way.
\item[(iii)] Since $C^*(U,I)\sim C(\bt)$, any continuous function $f:\bt\rightarrow\bt$ defines a unitary 
$$
f(U)=:u\in C^*(U,I)\subset\ba_\a\,,
$$
and vice-versa. Therefore, the Anzai skew-product in Definition \ref{nazai2}
is equivalently associated with the angle $\th$ and the unitary $u:=f(U)$.
\end{itemize}
\end{rem}
For $f\in C(\bt;\bt)$, define
\begin{equation}
\label{coo}
f_{n}(z):=\left\{\begin{array}{ll}
                      \prod_{l=0}^{n-1}f\circ R_{-l\a}(z)\,,   \,\,\,
                      n>0\,,& \\[1ex]
                       1\,,\,\,\, n=0\,,&\\[1ex]
                     \prod_{l=1}^{|n|}\overline{f\circ R_{l\a}(z)}\,,  \,\,\,
                      n<0\,.&
                    \end{array}
                    \right.
\end{equation}
\begin{prop}
\label{nz1}
The canonical trace $\tau$ is invariant under the action of any Anzai skew-product $\F_{\th,f}$.
\end{prop}
\begin{proof}
We start by noticing that 
$$
\F_{\th,f}(U^mV^n)=\sum_{k\in\bz}\widehat{f_{m,n}}(k)U^kV^n\,,
$$
where the $f_{m,n}\in C(\bt;\bt)$ are given for $m,n\in\bz$ by
$$
f_{m,n}(z):=e^{\imath m\th}z^m f_n(z)\,, \quad z\in\bt\,. 
$$
Therefore, $\t\left(\F_{\th,f}(U^mV^n)\right)=\widehat{f_{m,0}}(0)$. Since $f_{m,0}(z)=e^{\imath m\th}z^m$, and thus $\widehat{f_{m,0}}(0)=\d_{m,0}$, the assertion follows by a standard approximation argument because 
$$
\t\left(\F_{\th,f}(U^mV^n)\right)=\d_{m,0}\d_{n,0}=\t(U^mV^n)\,,\quad m,n\in\bz\,.
$$
\end{proof}

\section{ergodic properties of the anzai skew-product}
\label{sec4}

For a unital $*$-automorphism $\g$ of a unital $C^*$-algebra $\ga$, $\cs(\ga)^\g$ and $\partial\left(\cs(\ga)^\g\right)$ denote the convex, $*$-weakly compact set of the invariant states of $\ga$ under the action of $\g$, and its convex boundary made of extreme invariant states, respectively. The elements of $\partial\left(\cs(\ga)^\g\right)$ are called the {\it ergodic states}.

Let $\f\in\cs(\ga)^\g$ and $E^{\f,\g}_1\in\cb(\ch_\f)$ be the self-adjoint projection onto the eigenspace of the vectors invariant under $V_{\f,\g}$. We recall that 
$$
\dim\big(E^{\f,\g}_1\big)=1 \Longrightarrow \f\in\partial\left(\cs(\ga)^\g\right)\,,
$$
see {\it e.g.} \cite{Sa}, Proposition 3.1.10.

We set $M:=\pi_\tau(\ba_\alpha)''$ the von Neumann algebra generated by the GNS representation of the canonical trace $\tau$. We note that:
\begin{itemize}
\item[(i)] $M$ acts in standard form (with conjugation $J_\t$) on $\ch_\t$,
\item[(ii)] $\ker(\pi_\t)=\{0\}$ and therefore $\pi_\t$ realises a $*$-automorphism between $\ba_\a$ and its image $\pi_\t(\ba_\a)\subset M$,
\item[(iii)] taking into account Proposition \ref{nz1}, the powers of the adjoint action $\ad_{V_{\t,\F_{\th,f}}}$ extend faithfully the action associated with the Anzai skew-product $\F_{\th,f}$ to all of 
$M$:
$$
\pi_\t\circ \F_{\th,f}^n\circ\pi_\t^{-1}=\ad\!{}_{V^n_{\t,\F_{\th,f}}}\lceil_{\pi_\t(\ba_\a)}\,,\quad n\in\bz\,.
$$
\end{itemize}
From now on, for the Anzai skew-product $\F_{\th,f}$, we assume that $\th/2\pi$ is irrational without further mention.

We are now in a position to state and prove our theorem that says exactly when a quantum Anzai flow $\F_{\th,f}$ is ergodic. 
\begin{thm}
\label{thm:ergodic}
Let $f\in C(\bt;\bt)$, and $f_n$ defined in \eqref{coo}. For any $\a\in\br$, the $C^*$-dynamical system $\big(\ba_\alpha, \F_{\th,f},\tau\big)$ is ergodic if and only if, for each $n\in\mathbb{Z}\smallsetminus\{0\}$, the equation
\begin{equation}
\label{eq:cohom}
\big(g\circ R_{\th/2\pi}\big)f_n=g
\end{equation}
has no non-zero solution $g\in L^2(\mathbb{T},{\rm m})$.
\end{thm}
\begin{proof}
Fix $g\in L^2(\mathbb{T},{\rm m})$, and define $g\big(\pi_\t(U)\big)$ through the measurable functional calculus. It is a closed, in general unbounded operator acting on $\ch_\t$ which does not depend on the measurable representative function $g$. By a standard approximation argument, it is matter of routine to verify that, first $\pi_\t(V)^n\xi_\t\in\cd_{g(\pi_\t(U))}$ for each $n\in\bz$, and second
$$
\ch_\t=\bigg\{\sum_{n\in\bz}g_n\big(\pi_\t(U)\big)\pi_\t(V)^n\xi_\t\mid g_n\in L^2(\mathbb{T},{\rm m})\,\text{s.t.}\,\sum_{n\in\bz}\|g_n\|^2<+\infty\bigg\}\,.
$$
In addition, 
$$
\ch_\t\ni\xi=\sum_{n\in\bz}g_n\big(\pi_\t(U)\big)\pi_\t(V)^n\xi_\t=0\iff g_n=0\,,\,\,\,n\in\bz\,.
$$
Indeed, by the following simple calculations
\begin{align*}
\|\xi\|^2=&\sum_{m,n\in\bz}\big\langle g_m\big(\pi_\t(U)\big)^*g_n\big(\pi_\t(U)\big)\pi_\t(V^{n-m})\xi_\t,\xi_\t\big\rangle\\
=&\sum_{m,n\in\bz}\big\langle g_m\big(\pi_\t(U)\big)^*g_n\big(\pi_\t(U)\big)\xi_\t,\xi_\t\big\rangle\d_{m,n}\\
=&\sum_{n\in\bz}\big\langle g_n\big(\pi_\t(U)\big)^*g_n\big(\pi_\t(U)\big)\xi_\t,\xi_\t\big\rangle\\
=&\sum_{n\in\bz}\int|g_n(z)|^2\di m(z)=\sum_{n\in\bz}\|g_n\|^2\,,
\end{align*}
we obtain the assertions by taking into account that elements of the form $g\big(\pi_\t(U)\big)\pi_\t(V)^n\xi_\t=\pi_\t\big(g(U)V^n\big)\xi_\t$ for $g\in C(\bt)$ and $n\in\bz$ generate a total set of 
$\ch_\t$. Obviously, $\xi_\t$ corresponds to the singleton $g_n(z)=\d_{n,0}$.

As explained below, $\big(\ba_\alpha, \F_{\th,f},\tau\big)$ is ergodic (or in other words $\t$ is extreme among the invariant states under the action of $ \F_{\th,f}$) if $E^{\f,\F}_1\ch_\t=\bc\xi_\t$.

Taking into account the above considerations, we expand a generic vector in $\ch_\t$ as $\xi=\sum_{n\in\bz}g_n\big(\pi_\t(U)\big)\pi_\t(V)^n\xi_\t$. If $n=0$, the equation 
$V_{\t,\F_{\th,f}}\xi=\xi$ leads to $g_0\circ R_{\th/2\pi}=g_0$ in $L^2(\mathbb{T},{\rm m})$, which has the unique solution $g_0=1$ because $\th$ is irrational. Therefore, we can reduce the matter to 
sums of the form $\xi=\sum_{n\neq0}g_n\big(\pi_\t(U)\big)\pi_\t(V)^n\xi_\t$ which generate $\ch_\t\ominus\{\bc\xi_\t\}$.
We obtain
\begin{align*}
0=&V_{\t,\F_{\th,f}}\xi-\xi\\
=&\sum_{n\neq0}\bigg(g_n\big(\pi_\t\big(\F_{\th,f}(U)\big)\big)\pi_\t\big(\F_{\th,f}(V^n)\big)\xi_\t-g_n\big(\pi_\t(U)\big)\pi_\t(V^n)\xi_\t\bigg)\\
=&\sum_{n\neq0}\bigg(g_n\big(\pi_\t(e^{\imath\th}U)\big)\pi_\t(f_n(U))-g_n\big(\pi_\t(U)\big)\bigg)\pi_\t(V^n)\xi_\t\\
=&\sum_{n\neq0}\bigg(g_n\big(\pi_\t(e^{\imath\th}U)\big)f_n\big(\pi_\t(U)\big)-g_n\big(\pi_\t(U)\big)\bigg)\pi_\t(V^n)\xi_\t\\
=&\sum_{n\neq0}\bigg(\big((g_n\circ R_{\th/2\pi})f_n\big)\big(\pi_\t(U)\big)-g_n\big(\pi_\t(U)\big)\bigg)\pi_\t(V^n)\xi_\t\,,
\end{align*}
which leads to \eqref{eq:cohom} has no nonzero $L^2$-solutions for each $n\neq0$
\end{proof}
\begin{rem}
\label{md1}
We note that, if $g$ is any measurable function finite almost everywhere, and satisfying $g\big(e^{\imath\th}z\big)f_n(z)=g(z)$ a.e. for some $n\in\bz$ and $\th\in\br$ with $\th/2\pi$ irrational, then $|g(z)|={\rm const.}$ a.e.\,.
\end{rem}
By taking the absolute value, we have $\big|g(e^{\imath\th}z)f_n(z)\big|=|g(z)|$, which leads to $\big|g(e^{\imath\th}z)\big|=|g(z)|$ because $|f_n(z)|=1$. Therefore, 
$\big|g(e^{\imath n\th}z)\big|=|g(z)|$ for each $n\in\bz$, which leads to the assertion thanks to ergodicity.

\noindent\hspace{12.33 cm}$\square$

We note that, by Remark \ref{md1}, any solution of the cohomological equation \eqref{eq:cohom} belongs to $L^\infty(\bt,{\rm m})$.

With an argument similar to the one leading to Theorem \ref{thm:ergodic}, we have the following result which characterises the topological ergodicity for the Anzai skew-product in terms of the absence of \emph{continuous} solutions of the cohomological equations \eqref{eq:cohom}.
\begin{prop}
\label{thm:topergodic}
For the $C^*$-dynamical system $(\mathbb{A}_\alpha, \Phi_{\theta, f})$, we have $(\ba_\a)^{\Phi_{\theta, f}}=\bc\idd_{\ba_\a}$ if and only if, for each $n\in\mathbb{Z}\smallsetminus\{0\}$, the equation \eqref{eq:cohom} has no non-zero solution $g\in C(\mathbb{T})$.
\end{prop}
\begin{proof}
For $x\in\ba_\a\smallsetminus  C^*(U,I)$, consider the expansion in \eqref{exfp} given by
$x=\sum_{n\neq0}c_{n, x}(U)V^n$. Then
$$
\Phi_{\theta, f}(x)-x=\sum_{n\neq0}\big(c_{n, x}(e^{\imath\th}U)f_n(U)-c_{n, x}(U)\big)V^n\,,
$$
where the sum converges in norm in the sense previously explained, and the $f_n$ are given in \eqref{coo}. Therefore, $(\ba_\a)^{\Phi_{\theta, f}}\supsetneq\bc\idd_{\ba_\a}$ if and only if $\Phi_{\theta, f}(x)-x=0$ for some nonzero element $x\in\ba_\a$ as above. But this happens if and only if, for some $n\neq0$, $c_{n, x}(e^{\imath\th}U)f_n(U)=c_{n, x}(U)$ holds true for a non-zero continuous function $c_{n, x}$, or equivalently if and only if the equation \eqref{eq:cohom} has some non-zero solution $g\equiv c_{n, x}\in C(\mathbb{T})$, $n\neq0$.
\end{proof}
The next proposition provides many examples of ergodic non-com\-mutative Anzai skew-products, generalising the well-known ones arising from classical ergodic theory, see {\it e.g.}
\cite{A}, pag. 84.
\begin{prop} \label{prop:examples}
If $f^{(z_o)}(z)=z_oz$, where $z_o\in\bt$, then the corresponding  dynamical system
$(\mathbb{A}_\alpha, \Phi_{\theta, f^{(z_o)}},\tau)$ is ergodic for any $\alpha\in \mathbb{R}$.
\end{prop}
\begin{proof}
Thanks to Theorem \ref{thm:ergodic}, all we need to show is the cohomological equations for $f^{(z_o)}$ have no nonzero $L^\infty$-solutions for $n\neq0$. Setting
$$
a_{n}:=\left\{\begin{array}{ll}
                      e^{-\pi\imath n(n-1)\a}\,,   \,\,\,
                      n>0\,,& \\
                     e^{\pi\imath n(|n|+1)\a}\,,  \,\,\,
                      n<0\,,&
                    \end{array}
                    \right.
$$
for $f^{(z_o)}_n$ in \eqref{coo} we get $f^{(z_o)}_n(z)=a_n(z_oz)^n$. 

Fix now $g\in L^\infty(\mathbb{T}, {\rm m})\subset L^2(\mathbb{T}, {\rm m})$ given by
$g(z)=\sum_{k\in\bz}\widehat{g}(k)z^k$, where $\{\widehat{g}(k)\}_{k\in\bz}$ is a square-summable double sequence. For each fixed $n\neq0$, \eqref{eq:cohom} leads to
$$
\bigg(a_nz_o^n\widehat{g}(k)e^{\imath k\th}-\widehat{g}(k+n)\bigg)=0\,.
$$
Consequently, we easily obtain $\left|\widehat{g}(k)\right|=\left|\widehat{g}(k+n)\right|$ and therefore $\widehat{g}(k)=0$ for all $k\in\bz$, otherwise
$\widehat{g}(k)\nrightarrow0$ which would be a contradiction because of the square-summability of the $\widehat{g}(k)$. This means that $g$ is the null-function in $L^2(\bt,{\rm m})$. Therefore, the equation \eqref{eq:cohom} cannot have any non-null solution for each $n\in\bz\smallsetminus\{0\}$.
\end{proof}
We are now ready to prove one of the main results concerning unique ergodicity. It can be viewed as a noncommutative version of Lemma 2.1 in  \cite{Fu} involving skew-products for processes on the torus ({\it i.e.} a natural generalisation of Anzai skew-products).
\begin{thm}
\label{thm:uniquelyergodic}
For each $\a\in\br$, let $f\in C(\bt;\bt)$ such that the $C^*$-dynamical system $\big(\ba_\alpha, \F_{\th,f},\tau\big)$ is ergodic. Then  $\big(\ba_\alpha, \F_{\th,f}\big)$ is uniquely ergodic with the canonical trace $\t$ as the unique invariant state.
\end{thm}
\begin{proof}
We already know that, under the hypotheses made in the theorem, the trace $\tau$ is an extreme state in $\mathcal{S}(\mathbb{A}_\alpha)^{\Phi_{\theta, f}}$.
What we now need to do is show that the trace $\t$ is in addition the only $\Phi_{\theta, f}$-invariant state. 

For the Anzai skew-product $\F_{\th,f}$ and $\r^{(o)}_{1,z}$ given in \eqref{rozero}, it is immediate to see that $\r^{(o)}_{1,z}\circ\F_{\th,f}=\F_{\th,f}\circ\r^{(o)}_{1,z}$. Consequently, 
$\om\in\mathcal{S}(\mathbb{A}_\alpha)^{\Phi_{\theta, f}}\Rightarrow\om\circ\r^{(o)}_{1,z}\in\mathcal{S}(\mathbb{A}_\alpha)^{\Phi_{\theta, f}}$.

Let now $\omega$ be any invariant state. Since the irrational rotations on the $1$-dimensional torus are well known to be uniquely ergodic, we immediately see that 
$\omega(U^m)=\delta_{m,0}$. 

We next consider a new state $\bar{\omega}$, which is obtained by averaging the given $\omega$ under the action of the $*$-automorphisms $\r^{(o)}_{1,z}$ defined above:
$\bar{\omega}:=\oint\frac{\di z}{2\pi\imath z}\omega\circ\r^{(o)}_{1,z}$. We claim that $\bar{\omega}$ is always the trace $\tau$ irrespective of what
$\omega$ actually was.
Indeed, for the collection of linear generators $\big\{U^mV^n\big\}_{m.n\in\bz}$, one has
\begin{align*}
\bar{\omega}(U^mV^n)=&\oint\omega\big(\r^{(o)}_{1,z}(U^mV^n)\big)\frac{\di z}{2\pi\imath z}=\omega(U^mV^n)\oint\frac{z^n\di z}{2\pi\imath z}\\
=&\omega(U^mV^n)\delta_{n,0}=\omega(U^m)\delta_{n,0}=\delta_{m,0}\delta_{n,0}=\t(U^mV^n)\,.
\end{align*}
The last step is to note that $\om$ should be the trace itself because we have just shown that $\t$ is obtained by averaging a class of invariant states. For such a purpose, suppose $\om\neq\t$ and define
$$
\f_o:=\frac1{\pi}\int_0^{\pi}\omega\circ\r^{(o)}_{1,e^{\imath\b}}\di\b\,,\quad \psi_o:=\frac1{\pi}\int_{\pi}^{2\pi}\omega\circ\r^{(o)}_{1,e^{\imath\b}}\di\b\,.
$$
If neither were the trace $\t$, we immediately would be lead to a contradiction because $\t=\frac12(\f_o+\psi_o)$, being extreme, cannot be obtained by convex combination of invariant states. We then conclude that both $\f_o$ and $\psi_o$ should be the trace $\t$. This means 
$\t=\frac1{\pi}\int_0^{\pi}\omega\circ\r^{(o)}_{1,e^{\imath\b}}\di\b$. We can repeat the process, obtaining a sequence of invariant states $\{\f_j\}_{j\in\bn}$ with
$\f_j:=\frac{2^j}{\pi}\int_0^{\pi/2^j}\omega\circ\r^{(o)}_{1,e^{\imath\b}}\di\b$. If some of them were not the trace, we would obtain again a contradiction by extremality as before, and we are done. 
Conversely, if all of them coincided with the trace, taking the limit in the $*$-weak topology, we would easily obtain
$$
\t=\frac{2^j}{\pi}\int_0^{\pi/2^j}\omega\circ\r^{(o)}_{1,e^{\imath\b}}\di\b=\lim_j\frac{2^j}{\pi}\int_0^{\pi/2^j}\omega\circ\r^{(o)}_{1,e^{\imath\b}}\di\b=\om\,,
$$
that is any invariant state $\om$ coincides with the trace $\t$.
\end{proof}
Collecting together Theorems \ref{thm:ergodic} and \ref{thm:uniquelyergodic}, we have the following
\begin{cor}
For an Anzai skew-product $\F_{\th,f}$ on $\ba_\a$, the following are equivalent:
\begin{itemize}
\item[(i)] the canonical trace $\t$ is ergodic for $\F_{\th,f}$,
\item[(ii)] the equation $\big(g\circ R_{\th/2\pi}\big)f_n=g$ has no nonzero solution in $L^\infty(\bt,{\rm m})$ for each $n\in\bz\smallsetminus\{0\}$,
\item[(iii)] the dynamical system $\big(\ba_\a, \F_{\th,f}\big)$ is uniquely ergodic.
\end{itemize}
\end{cor}
\noindent\hspace{12.33 cm}$\square$

Concerning the minimality, we have the following
\begin{cor}
Suppose that $\big(\ba_\alpha, \F_{\th,f},\tau\big)$ is ergodic. Then $\big(\ba_\alpha, \F_{\th,f}\big)$ is minimal.
\end{cor}
\begin{proof}
Since the canonical trace $\t$ is faithful, the assertion follows by \cite{LongoPel}, Corollary 2.6.
\end{proof}
It is worth pointing out that ,in the classical setting, examples are known of Anzai skew-products which are minimal but nevertheless fail to be uniquely ergodic, see {\it e.g.} \cite{Fu, R} and the references therein. To the best of our knowledge, examples of this sort have not yet been exhibited in the noncommutative setting. Therefore, we intend to return to this interesting point elsewhere. 

\section{processes on the torus}
\label{sec5}

By following the strategy of the proof of Theorem \ref{thm:uniquelyergodic}, we can extend the analysis in Section 2 of \cite{Fu} relative to the so-called processes on the torus, to the non-separable cases.

Indeed, consider a classical dynamical system $\big(C(X_o), T_o,\m_o\big)$, where $X_o$ is a compact space, $\m$ is a probability Radon measure on $X_o$, and finally 
$T_o:X_o\rightarrow X_o$ 
is a homeomorphism preserving the measure $\m_o$. 

On the cartesian product $X:=X_o\times \bt$, consider the (generalisation of the Anzai) skew-product 
$T(x_o,z):=\big(T_ox_o,f(x_o)z\big)$ where $f:X_o\rightarrow\bt$ is a continuous function. As explained in \cite{Fu}, Section 2, it is easily seen that $T$ leaves the measure $\m:=\m_o\times{\rm m}$ invariant, and so we can naturally deal with $\big(C(X), T,\m\big)$. 

The following result generalises Lemma 2.1 of \cite{Fu} to the case when the probability measure $\m_o$ on $X_o$ is not standard.\footnote{The positive Radon measure $\m\in C(X)^*$ on the compact Hausdorff space $X$ is said to be {\it standard} if $L^2(X,\m)$ is separable.}
\begin{prop}
Suppose that the dynamical system $\big(C(X_o), T_o\big)$ is uniquely ergodic. For the dynamical system $\big(C(X), T,\m\big)$, the following are equivalent:
\begin{itemize}
\item[(i)] $\m$ is an ergodic measure for the action of $T$,
\item[(ii)] the equation 
$$
g(T_ox_o)f(x_o)^n=g(x_o)\,, \,\,\m_o\text{-}\,{\rm a.e.}\,,
$$ 
has no nonzero solution in $L^\infty(X_o,\m_o)$ for each $n\in\bz\smallsetminus\{0\}$,
\item[(iii)] the dynamical system $\big(C(X), T\big)$ is uniquely ergodic.
\end{itemize}
\end{prop}
\begin{proof}
(i)$\iff$(ii) follows similarly to Lemma 2.1 of \cite{Fu}, or equally well to Theorem \ref{thm:ergodic}, by expanding a generic function $h\in L^2(X_o\times\bt,\m_o\times{\rm m})$  in orthogonal pieces as
$$
h(x_o,z)\sim\sum_{n\in\bz}h_n(x_o)z^n\,,
$$
with $\{h_n\}_{n\in\bz}\subset L^2(X_o,\m_o)$ such that $\sum_{n\in\bz}\|h_n\|^2<+\infty$.

(iii)$\Rightarrow$(i) is trivial by noticing that $\m$ is invariant under the action of $T$.

(i)$\Rightarrow$(iii) follows the same line of Theorem \ref{thm:uniquelyergodic} by considering the point-transformation $S_\b$ defined as in \cite{Fu} by
$$
S_\b(x_o,z):=(x_o,e^{\imath\b}z)\,,\quad x_o\in X_o\,,\,\,z\in\bt\,,\,\,\b\in[0,2\pi)\,.
$$
We first notice that $\n\circ S_\b$ is invariant under the action of $T$, provided $\n$ is invariant. Therefore, if $\n$ is any invariant probability measure and $\om:=\int_Xf\di\n$ is the corresponding invariant state on $C(T)$, then $\bar\om:=\frac{1}{2\pi}\int_0^{2\pi} \omega\circ S_\beta \textrm{d}\beta$ is an invariant state which cannot be extreme as we have shown in the proof of Theorem \ref{thm:uniquelyergodic}. We then obtain the assertion if we show that $\bar\om(f)=\int_Xf\di\m=\om(f)$ for each $f\in C(X)$. By the Stone-Weierstrass theorem, it is enough to consider functions of the form $f(x_o,z)=F(x_o)z^n$ which generate a dense set in $C(X)$ as $F$ runs in $C(X_o)$ and $n$ in $\bz$. 

For such a purpose, we start by noticing that for functions $f_F(x_o,z):=F(x_o)$ depending only on the first variable $x_o$, $\n(f_F)$ defines an invariant state on $C(X_o)$ which must coincide with that associated with $\m_o$ by unique ergodicity: $\int_X f_F\di\n=\int_{X_o}F\di\m_o$. 

Now we compute
\begin{align*}
\bar\omega(f)=&\frac{1}{2\pi}\int_0^{2\pi}\di\b e^{\imath n\beta}\int_{T_o\times\bt}\di\n(x_o,z)F(x_o)z^n\\
=&\d_{n,0}\int_{X}f_F\di\n=\d_{n.0}\int_{X_o}F\di\m_o\\
=&\int_X f\di\m =\om(f)\,,
\end{align*}
which ends the proof.
\end{proof}

\section{on the convergence of the average $\frac1{n}\sum_{k=0}^{n-1}\l^{-k}\F^k$}
\label{sec6}

The present section is devoted to showing that there exists an Anzai skew-product on $\ba_\a$ for which the Cesaro means $M_{a,\l}$ given in \eqref{cesmn} do not converge in the weak topology of $\mathbb{A_\alpha}$, for some $a\in\mathbb{A}_\alpha$ and $\l\in\bt$. We then provide examples which do not arise by the tensor product construction exhibited in \cite{F4}.

The following lemma, due to Furstenberg ({\it cf.} \cite{Fu}, pag. 585) and reported in \cite{R}, Lemma 3.2,  applies to the cohomological equation \eqref{eq:cohom}
for $n=1$. 
\begin{lem}
\label{lem:fur1}
There is an angle $\theta$ such that $\frac{\theta}{2\pi}\notin\mathbb{Q}$ and a function $f\in C(\mathbb{T})$ such that the associated cohomological equation \eqref{eq:cohom} 
with $n=1$ admits a solution $g\in L^\infty(\mathbb{T},{\rm m})\smallsetminus C(\mathbb{T})$.
\end{lem}

\noindent\hspace{12.33 cm}$\square$

The following result addresses non-ergodic Anzai skew-products, which cannot be directly obtained by the known results, {\it cf.} \cite{AD}, Theorem 3.2.
\begin{lem}
\label{lem:fur2}
Let $\theta$ be such that $\frac{\theta}{2\pi}\notin\mathbb{Q}$ and $f\in C(\mathbb{T})$. If for some $n\in\mathbb{Z}\smallsetminus\{0\}$, the cohomological equation \eqref{eq:cohom} 
associated with the Anzai skew-product $(\mathbb{A}_\alpha, \Phi_{\theta, f}, \tau)$ admits a solution $g\in L^\infty(\mathbb{T},{\rm m})\smallsetminus C(\mathbb{T})$, then the limit of the Cesaro means
$M_{a,1}$ given in \eqref{cesmn} does not exist in the weak topology for some $a\in \mathbb{A}_\alpha$.
\end{lem}
\begin{proof}
We put $\F:=\Phi_{\theta, f}$. By using the fact that $g\in L^\infty(\bt,{\rm m})$ is a non-continuous nonzero solution of the cohomological equation for some $n\neq0$, and in addition that $g\big(\pi_\t(U)\big)$ belongs, up to a normalisation, to the unitary group of the abelian von Neumann algebra 
$\pi_\t\big(C^*(U,I)\big)''$, see {\it e.g.} Remark \ref{md1}, we compute
\begin{equation*}
\begin{split}
\pi_\t\big(\Phi(V^n)\big)=&\pi_\t\big((f(U)V)^n\big) = \pi_\t\big(f_n(U)V^n\big)\\
=&g\big(\pi_\t(e^{\imath \th}U)\big)^{-1}g\big(\pi_\t(U)\big)\pi_\t(V^n)\,.
\end{split}
\end{equation*}
This leads to
$$
\pi_\t\big(\F(V^n)V^{-n}\big)=g\big(\pi_\t(e^{\imath \th}U)\big)^{-1}g\big(\pi_\t(U)\big)\,.
$$
Now we claim that, for each $k\in\bn$,
\begin{equation}
\label{ineq}
\pi_\t\big(\F^k(V^n)V^{-n}\big)=g\big(\pi_\t(e^{\imath k\th}U)\big)^{-1}g\big(\pi_\t(U)\big)\,,
\end{equation}
which is proven by induction.
Indeed, it is trivially true for $k=0$, and the previous calculation has shown that it holds true also when $k=1$. 

Now we assume it holds true for a generic integer  $k\in\bn$ and deduce that it holds true for $k+1\in\bn$. We have
\begin{align*} 
&\pi_\t\big(\F^{k+1}(V^n)V^{-n}\big)=\pi_\t\big(\F\big(\F^k(V^n)V^{-n}\big)\big)\pi_\t\big(\F(V^n)V^{-n}\big)\\
=&\ad\!{}_{V_{\t,\Phi}}\left(g\big(\pi_\t(e^{\imath k\th}U)\big)^{-1}g\big(\pi_\t(U)\big)\right)
g\big(\pi_\t(e^{\imath \th}U)\big)^{-1}g\big(\pi_\t(U)\big)\\
=&g\big(\pi_\t(e^{\imath (k+1)\th}U)\big)^{-1}g\big(\pi_\t(e^{\imath \th}U)\big)g\big(\pi_\t(e^{\imath \th}U)\big)^{-1}g\big(\pi_\t(U)\big)\\
=&g\big(\pi_\t(e^{\imath (k+1)\th}U)\big)^{-1}g\big(\pi_\t(U)\big)\,.
\end{align*} 
Let $C(\bt)\ni h\mapsto h(U)\in\ba_\a$ be the map realising the $*$-isomorphism $C(\mathbb{T})\sim C^*(U,I)$, and fix 
$h\in C(\mathbb{T})$ such that $\int_{\mathbb{T}} h g^{-1}\di{\rm m} \neq 0$. By \eqref{ineq}, the representative in $\pi_\t(\ba_\a)$ of $\Phi^k(h(U)V^n)V^{-n}$ can be easily computed as 
\begin{align*} 
\pi_\t\big(\Phi^k(h(U)V^n)V^{-n}\big)=&\pi_\t\big(\Phi^k(h(U)\big)\pi_\t\big(\Phi^k(V^n)V^{-n}\big)\\
=&\pi_\t(h(e^{\imath k\th}U))g(\pi_\t(e^{\imath k\th}U))^{-1}\big)g\big(\pi_\t(U)\big)\,.
\end{align*} 
The above calculation leads to
\begin{align*} 
\pi_\t(h(e^{\imath k\th}U))g(\pi_\t(e^{\imath k\th}&U))^{-1}g\big(\pi_\t(U)\big)
=\pi_\t\big(\Phi^k(h(U)V^n)V^{-n}\big)\\
\in&\pi_\t(\ba_\a)\bigcap\pi_\t\big(C^*(U,I)\big)''=\pi_\t\big(C^*(U,I)\big)\,,
\end{align*} 
and therefore, for each $k\in\bn$ there exist a unique function $G_k\in C(\bt)$ such that
$$
\Phi^k(h(U)V^n)V^{-n}=G_k(U)\,.
$$
In addition, the last computations also yield
$$
G_k(z)=h(e^{\imath k\th} z)g(e^{\imath k\th} z)^{-1}g(z)\,, \,\,\text{m-a.e.}\,.
$$
Now, for each $z\in\bt$ we consider any extension $\om_z\in\cs(\ba_\a)$ of the state on $C^*(U,I)\sim C(\bt)$ corresponding to the evaluation at the point $z$ ({\it i.e.} $\om_z(f(U))=f(z)$ for $f\in C(\bt)$), together with $\f_z:=\om_z(\,{\bf\cdot}\,V^{-n})$. Notice that
\begin{equation*}
\begin{split}
\f_z\big(\Phi^k(h(U)V^n)\big)=&\om_z\big(\Phi^k(h(U)V^n)V^{-n}\big)\\
=&G_k(z)=h(e^{\imath k\th} z)g(e^{\imath k\th} z)^{-1}g(z)\,,
\end{split}
\end{equation*}
where the last equality is understood almost everywhere.

Concerning the Cesaro averages, we compute 
\begin{align*} 
\frac1{l}\sum_{k=0}^{l-1}G_k(z)=&\f_z\left(M_{1,h(U)V^n}(l)\right)\\
=&\frac1{l}\sum_{k=0}^{l-1}h(e^{\imath k\th} z)g(e^{\imath k\th} z)^{-1}g(z)\,,
\end{align*} 
where as usual the last equality holds almost everywhere.

Notice that, by an elementary application of the Birkhoff individual ergodic theorem,
\begin{align*} 
&\lim_l\bigg(\frac1{l}\sum_{k=0}^{l-1}h(e^{\imath k\th} z)g(e^{\imath k\th} z)^{-1}g(z)\bigg)\\
=&\lim_l\bigg(\frac1{l}\sum_{k=0}^{l-1}h(e^{\imath k\th} z)g(e^{\imath k\th} z)^{-1}\bigg)g(z)
=Cg(z)
\end{align*}
almost everywhere, where $C:=\int_{\mathbb{T}} h g^{-1}\di{\rm m}\neq0$. 

If $M_{a,1}$ were convergent in the weak topology for each $a\in\ba_\a$, we would obtain
$$
g(z)=\lim_l\frac{\frac1{l}\sum_{k=0}^{l-1}G_k(z)}{C}\,,
$$ 
almost everywhere. Therefore, the measurable non-continuous function $g$ would be in the equivalence class of the pointwise limit of the sequence of continuous functions $\frac{\frac1{l}\sum_{k=0}^{l-1}G_k(z)}{C}$, that is it would coincide with a Baire $1^{\rm st}$ class function, almost everywhere. As shown in \cite{Fu}, pag. 584, this leads to a contradiction because, in our situation, $g$ must coincide with a continuous function, almost everywhere, which is contrary to our assumption.
\end{proof}
Lemma \ref{lem:fur1} says that there might be Anzai skew-products $\Phi_{\theta, f}$ that are topologically ergodic,
{\it i.e.} $(\ba_\a)^{\Phi_{\theta, f}}=\bc\idd_{\ba_\a}$, but not ergodic {\it i.e.} $\dim\big(E^{\t,\Phi_{\theta, f}}_1\ch_\t\big)>1$. 
Indeed, by Proposition \ref{thm:topergodic}, it would be enough to determine $f$ for which the cohomological equation \eqref{eq:cohom} does not admit any continuous solution but the trivial ones $\bc1$ for $n=0$, and in addition has some nontrivial measurable non-continuous solution for some $n\in\bz\smallsetminus\{0\}$.

On the other hand, Lemma \ref{lem:fur2} says that there are Anzai skew-products which do not enjoy either a weaker property of unique ergodicity, which we are going to define.

A $C^*$-dynamical system $(\ga,\F)$ based on the unital $C^*$-algebra and the identity-preserving completely positive map $\F:\ga\rightarrow\ga$ is said to be {\it uniquely ergodic w.r.t. the fixed point subalgebra} if it satisfies one of the properties (i)-(vi) listed in Theorem 2.1 of \cite{FM3}, see also \cite{AD}, Definition 3.3. 
In particular, $(\ga,\F)$ is uniquely ergodic w.r.t. the fixed point subalgebra if and only if the ergodic average $\frac1{n}\sum_{k=0}^{n-1}\F^k(a)$ converges, pointwise in norm, for each $a\in\ga$. 

An immediate consequence of Lemma \ref{lem:fur2} is the following
\begin{cor}
Under the hypotheses of Lemma \ref{lem:fur2}, the $C^*$-dynamical system $\big(\mathbb{A}_\alpha, \Phi_{\theta, f}\big)$ is not uniquely ergodic w.r.t. the fixed point subalgebra.
\end{cor}
\noindent\hspace{12.33 cm}$\square$

The following result is the goal of the present section. Namely, there are noncommutative uniquely ergodic $C^*$-dynamical systems $(\ga,\F,\f)$, 
not based on the tensor product construction, 
with  $\ga$ a unital $C^*$-algebra, and $\F:\ga\rightarrow\ga$ a $*$-automorphism, for which $\s_{\rm pp}(\F)\subsetneq\s_{\rm pp}(V_{\f,\F})$, such that for some 
$\l\in \s_{\rm pp}(V_{\f,\F})\smallsetminus\s_{\rm pp}(\F)$ and $a\in\ga$, the average $\frac1{n}\sum_{k=0}^{n-1}\l^{-k}\F^k(a)$ does not converge, even in the weak topology.
\begin{thm}
\label{thm:lambdacesarononconv-wt}
For each $\a\in\br$, there exists a uniquely ergodic Anzai skew-product $\big(\ba_\a,\F_{\th,\widetilde f},\t\big)$, with $\th$ irrational and $\widetilde f\in C(\bt;\bt)$ not depending on $\a$, such that 
$\s_{\rm pp}\big(\Phi_{\theta, \widetilde f}\big)\subsetneq\s_{\rm pp}\big(V_{\tau,\Phi_{\theta, \widetilde f}}\big)$, and the limit
\begin{equation}
\label{eq:lambdacesarolimit}
\lim_{n\to +\infty} \frac{1}{n} \sum_{k=0}^{n-1} \lambda^{-k} \Phi_{\theta,\widetilde f}^k(a)
\end{equation}
fails to exist in the weak topology of $\ba_\a$, for some $a\in\ba_\a$ and $\l\in \s_{\rm pp}\big(V_{\tau,\Phi_{\theta, \widetilde f}}\big)\smallsetminus\s_{\rm pp}\big(\Phi_{\theta, \widetilde f}\big)$.
\end{thm}
\begin{proof}
Fix $\n\in\br$ such that $\th$ and $\n$ are rationally independent.\footnote{A finite set $\{\n_1,\dots,\n_n\}\subset\br$ is said to be {\it rationally independent} if for any linear combination of the form
$\sum_{l=1}^nk_l\n_l=0$, with $k_l$ integers, implies $k_1=\cdots=k_n=0$.} 
We note that $\r^{(o)}_{e^{\imath\th},e^{\imath\n}}$ defined in \eqref{rozero} can be viewed as an Anzai skew product associated with the pair $\th,f_\n$ with $f_\n(z)=e^{\imath\n}$, identically. 
Since $\th$ and $\n$ are rationally independent, the cohomological equation \eqref{eq:cohom}, $n\neq0$,
has no non-zero solution in $L^\infty(\mathbb{T},{\rm m})$, and therefore the $C^*$-dynamical system $\big(\ba_\a, \r^{(o)}_{e^{\imath\th},e^{\imath\n}}\big)$ is uniquely ergodic with $\t$ the unique invariant state.

Denote by $\r_{e^{\imath\th},e^{\imath\n}}:\pi_\t(\ba_\a)''\rightarrow\pi_\t(\ba_\a)''$ the canonical extension of $\r^{(o)}_{e^{\imath\th},e^{\imath\n}}$ to the von Neumann algebra $\pi_\t(\ba_\a)''$. It is a 
$*$-automorphism preserving the tracial vector state 
$\langle\,{\bf\cdot}\,\xi_\t,\xi_\t\rangle$, such that the $W^*$-dynamical system $\big(\pi_\t(\ba_\a)'',\r_{e^{\imath\th},e^{\imath\n}},\langle\,{\bf\cdot}\,\xi_\t,\xi_\t\rangle\big)$ is ergodic.

Let $\theta$ and $f$ be as in the statement of  Lemma \ref{lem:fur1} and $\lambda := e^{\imath\nu} \in \mathbb{T}$.
Consider the function $\widetilde f(z) := e^{\imath\nu} f(z)$, and observe that $g\in L^\infty(\mathbb{T},{\rm m})\smallsetminus C(\mathbb{T})$ is a solution of the cohomological equation (\ref{eq:cohom}) for $n=1$, {\it i.e.} $g(e^{\imath\theta} z) f(z) = g(z)$ for almost every $z\in\mathbb{T}$, if and only if 
\begin{equation*}
%\label{eq:nucohom}
g(e^{i\theta} z) e^{-i\nu}\widetilde f(z) = g(z)\,,
\end{equation*}
for almost every $z\in\mathbb{T}$. 

Now we show that the Anzai skew-product $\Phi_{\theta, \widetilde f}$ is ergodic, and thus uniquely ergodic. 
To this aim, we define a $*$-automorphism $\r:\pi_\t(\ba_\a)''\rightarrow\pi_\t(\ba_\a)''$ realising the equivalence of the $W^*$-dynamical system 
$\big(\pi_\t(\ba_\a)'',\ad_{V_{\t,\Phi_{\theta, \widetilde f}}},\langle\,{\bf\cdot}\,\xi_\t,\xi_\t\rangle\big)$ with the ergodic $W^*$-dynamical system $\big(\pi_\t(\ba_\a)'',\r_{e^{\imath\th},e^{\imath\n}},\langle\,{\bf\cdot}\,\xi_\t,\xi_\t\rangle\big)$, which leads to the assertion.

Indeed, define $\r$ on the generators $u:=\pi_\t(U)$ and $v:=\pi_\t(V)$ 
$$
\r(u):=u\,,\quad \r(v):=g(u)^{-1}v\,.
$$
By the universal property of the noncommutative 2-torus, $\r$ provides a $*$-automorphism of $\pi_\t(\ba_\a)''$. In addition, an analogous computation as in Proposition \ref{nz1} leads to
$$
\langle\r(a)\xi_\t,\xi_\t\rangle=\langle a\xi_\t,\xi_\t\rangle\,,\quad a\in\pi_\t(\ba_\a)''\,,
$$
that is $\r$ preserves the tracial vector state. Finally, on the generators $u,v$ of $\pi_\t(\ba_\a)''$, 
\begin{align*}
&\r\big(\ad{}_{\!V_{\t,\Phi_{\theta, \widetilde f}}}(u)\big)=e^{\imath\th}u=\r_{e^{\imath\th},e^{\imath\n}}(\r(u))\,,\\
&\r\big(\ad{}_{\!V_{\t,\Phi_{\theta, \widetilde f}}}(v)\big)=\r(\widetilde f(u)v)=\widetilde f(u)g(u)^{-1}v\\
=&g\big(e^{\imath\th}u\big)^{-1}\big(e^{\imath\n}v\big)=\r_{e^{\imath\th},e^{\imath\n}}(\r(v))\,,
\end{align*}
which leads to $\r\circ\ad{}_{\!V_{\t,\Phi_{\theta, \widetilde f}}}=\r_{e^{\imath\th},e^{\imath\n}}\circ\r$.

Next, we show that $\lambda = e^{\imath\nu}$ is a measurable but not continuous eigenvalue of $\Phi_{\theta, \widetilde f}$ by exhibiting an eigenvector in $\ch_\t$ that does not come from an element of 
$\mathbb{A}_\alpha$. Indeed, consider $g(u)v\xi_\t\in\ch_\tau$ and  compute
$$
V_{\Phi_{\theta, \widetilde f}} g(u)v\xi_\tau = \ad{}_{\!V_{\t,\Phi_{\theta, \widetilde f}}}(g(u)v)\xi_\tau = g(e^{i\theta}u)\widetilde f(u)v\xi_\tau = e^{\imath\nu} g(u)v\xi_\tau\,,
$$ 
and thus $g(u)v\xi_\tau$ is an eigenvector with eigenvalue $\lambda$. 
Taking into account that
$$
\left\{g\big(\pi_\t(U)\big)V\mid g\in L^\infty(\bt,{\rm m})\right\}\bigcap\pi_\t(\ba_\a)=\left\{g\big(\pi_\t(U)\big)V\mid g\in C(\bt)\right\}\,,
$$
and $\xi_\t$ is separating, such an eigenvector is indeed a measurable non-continuous one.

Lastly, we show that the limit in \eqref{eq:lambdacesarolimit} with $\l=e^{\imath\n}$ fails to exist in the weak topology, for some $a\in\ba_\a$.
To this aim, consider the Anzai skew-product $\Phi_{\theta, f}$, which is not ergodic by assumption. For every $k\in\mathbb{N}$, we have 
$\Phi_{\theta, f}^k(U) = e^{\imath k\theta}U=\Phi_{\theta,\widetilde f}^k(U)$, while 
$\Phi_{\theta, f}^k(V) =\left[\prod_{l=0}^{k-1}f(e^{\imath l\theta}U)\right]V$ and $\Phi_{\theta, \widetilde f}^k(V) =\left[\prod_{l=0}^{k-1}\widetilde f(e^{\imath l\theta}U)\right]V=\lambda^k\Phi_{\theta,f}^k(V)$. Now take 
$h\in C(\mathbb{T})$ as in the proof of Lemma \ref{lem:fur2}. 
Comparing the two Anzai skew-products on $a := h(U)V\in \mathbb{A}_\alpha$, we get 
$$
\Phi_{\theta, \widetilde f}^k(h(U)V) = \Phi_{\theta, f}^k(h(U))\big(\lambda^k\Phi_{\theta, f}^k(V)\big) = \lambda^k \Phi_{\theta, f}^k(h(U)V)\,,
$$ 
and thus
$$
\frac{1}{n} \sum_{k=0}^{n-1}\lambda^{-k}\Phi_{\theta, \widetilde f}^k(h(U)V) = \frac{1}{n} \sum_{k=0}^{n-1} \Phi_{\theta, f}^k(h(U)V)\,.
$$
Since, by Lemma \ref{lem:fur2}, the r.h.s. fails to converge in the weak topology, the proof follows.
\end{proof}

\section*{Acknowledgements}

The present project is part of \lq\lq MIUR Excellence Department Project awarded to the Department of Mathematics, University of Rome Tor Vergata, CUP E83C18000100006". Furthermore, L. Giorgetti is supported by the European Union's Horizon 2020 research and innovation programme under Grant Agreement 795151 ``beyondRCFT'' H2020-MSCA-IF-2017.

\end{document}